\documentclass[11pt]{amsart}

\usepackage[T1]{fontenc}
\usepackage{amsmath}
\usepackage{amssymb}
\usepackage{mathtools}
\usepackage{enumitem}
\usepackage[colorlinks=true, linkcolor=cyan, citecolor=cyan]{hyperref}
\usepackage[capitalise]{cleveref}
\RemoveFromHook{label}[firstaid/cleveref]
\usepackage[margin=1.1in]{geometry}

\numberwithin{equation}{section}

\newtheorem{theorem}{Theorem}[section]
\newtheorem{proposition}[theorem]{Proposition}
\newtheorem{lemma}[theorem]{Lemma}
\newtheorem{corollary}[theorem]{Corollary}
\newtheorem{fact}[theorem]{Fact}
\newtheorem{conjecture}[theorem]{Conjecture}
\newtheorem*{theoremA}{Theorem A}
\newtheorem*{theoremB}{Theorem B}

\theoremstyle{definition}
\newtheorem{definition}[theorem]{Definition}
\newtheorem{remark}[theorem]{Remark}
\newtheorem{example}[theorem]{Example}
\newtheorem{examples}[theorem]{Examples}
\newtheorem{notation}[theorem]{Notation}

\newcommand{\C}{\mathbb{C}}
\newcommand{\R}{\mathbb{R}}

\newcommand{\Z}{\mathbb{Z}}
\newcommand{\Q}{\mathbb{Q}}

\newcommand{\ct}{\C(\!(t)\!)}
\newcommand{\rlt}{\R(\!(t)\!)}
\newcommand{\cL}{\mathcal{L}}
\newcommand{\cO}{\mathcal{O}}
\newcommand{\cM}{\mathcal{M}}
\newcommand{\cF}{\mathcal{F}}
\newcommand{\cU}{\mathcal{U}}
\newcommand{\KX}{K[X]}
\newcommand{\KrX}{K(X)}
\newcommand{\Lam}{\Lambda}
\newcommand{\tf}{T_\forall}
\newcommand{\Pres}{\mathrm{Pr}}
\newcommand{\Presf}{\mathrm{Pr}_\forall}
\DeclareMathOperator{\Th}{Th}
\DeclareMathOperator{\Frac}{Frac}

\DeclareMathOperator{\res}{res}

\DeclareMathOperator{\Spec}{Spec}
\DeclareMathOperator{\tSpec}{\textnormal{$T$-Spec}}
\DeclareMathOperator{\SpecT}{Spec_{T}}
\DeclareMathOperator{\Ker}{Ker}
\DeclareMathOperator{\tp}{tp}
\DeclareMathOperator{\ev}{ev}
\newcommand{\gam}{\gamma}
\newcommand{\rt}[2][]{\sqrt[\leftroot{2}\uproot{3}{#1}]{#2}}

\begin{document}

\title[A Nullstellensatz for complex Laurent series]{A Nullstellensatz for the field of complex Laurent series}

\author{Pablo Cubides Kovacsics}
\address{Departamento de Matemáticas, Universidad de los Andes, Carrera 1 \# 18A-12, Bogotá, Colombia}
\email{p.cubides@uniandes.edu.co}

\author{Sebastián Uribe Rincón}
\address{Departamento de Matemáticas, Universidad de los Andes, Carrera 1 \# 18A-12, Bogotá, Colombia}
\email{sa.uribe@uniandes.edu.co}

\subjclass[2020]{Primary 12L12, 12J10; Secondary 03C60, 14A05}
\keywords{Nullstellensatz, Laurent series, Kochen operator, valued fields, model completeness, spectra}

\begin{abstract}
Let $T$ be the theory of the valued field of complex Laurent series $\ct$ in the language of rings expanded by a constant for $t$ and by predicates for $n$-th powers. We introduce a $t$-adic analogue of Kochen's operator, namely $\gam(x)=x/(1+tx^{2})$, and the associated Kochen ring $\Lam_{L}$ of a field extension $L$ of a model $K$ of $T$. Our main result characterizes the field extensions of $K$ which embed into a model of $T$ (the \emph{formally $T$-adic} fields): $L$ is formally $T$-adic if and only if $1/t\notin\Lam_{L}$. As an application we prove a Nullstellensatz for $\ct$: a polynomial $f\in K[X_{1},\dots,X_{n}]$ vanishes on the zero set of an ideal $I$ if and only if $f^{m}=\sum_{i}\omega_{i}f_{i}$ for some $m\geqslant 1$, generators $f_{i}$ of $I$ and coefficients $\omega_{i}$ in the ring $\Lam_{K(X)}\cdot K[X]$. We derive the resulting correspondence between closed ideals and algebraic sets, and we compare natural incarnations of the associated $T$-spectrum: prime ideals, morphisms into models of $T$, and the space of types of $T$.
\end{abstract}

\maketitle

\section{Introduction}\label{sec:intro}

Hilbert's Nullstellensatz describes, for an algebraically closed field $K$ and an ideal $I$ of the polynomial ring $\KX=K[X_{1},\dots,X_{n}]$, the ideal of all polynomials vanishing on the zero set
\[
Z(I)=\{a\in K^{n}: f(a)=0 \text{ for all } f\in I\}
\]
as the radical $\sqrt{I}$. Over fields which are not algebraically closed the statement fails as stated, but for several important theories of fields it admits a precise substitute in which the radical is replaced by a notion tailored to the arithmetic of the base field. Two classical instances motivate the present work: the real and $p$-adic Nullstellensätze. For the real Nullstellensatz, let $K$ be a real closed field and $I\subseteq \KX$ an ideal. Then a polynomial $f$ vanishes on $Z(I)$ if and only if $f$ belongs to the real radical
\[
\sqrt[r]{I}=\Big\{g\in\KX : g^{2m}+\textstyle\sum_{i=1}^{l}h_{i}^{2}\in I \text{ for some } m\geqslant 1,\ h_{1},\dots,h_{l}\in\KX\Big\},
\]
see for instance \cite{bcr}. For the $p$-adic Nullstellensatz, let $K$ be a $p$-adically closed field with valuation ring $\cO_{K}$. Kochen \cite{Kochen69} introduced the operator
\[
\gam_{p}(x)=\frac{1}{p}\cdot\frac{x^{p}-x}{(x^{p}-x)^{2}-1},
\]
and the associated Kochen ring $\Lam=\big\{a/(1+pb): a,b\in \cO_K[\gam_{p}(K(X))]\big\}$. Jarden and Roquette \cite{jardenroquette} (see also Prestel--Roquette \cite{prestelroquette} and Srhir \cite{srhir, Srhir08}) proved that $f$ vanishes on $Z(I)$ if and only if $f^{m}=\sum_{i}\lambda_{i}f_{i}$ for some $m\geqslant 1$, generators $f_{i}$ of $I$ and $\lambda_{i}\in\Lam\cdot\KX$. In both cases the shape of the radical reflects a characterization of those field extensions of $K$ which embed into a model of the theory of $K$: formally real fields in the first case, formally $p$-adic fields in the second.

The purpose of this article is to develop the analogue, for $\ct$, of the Kochen operator, of the theory of formally $p$-adic fields, and of the $p$-adic Nullstellensatz.

Throughout, we work in the language
\[
\cL=\{+,-,\cdot,{}^{-1},0,1,t\}\cup\{P_{n}: n\geqslant 2\},
\]
where $t$ is a constant symbol and each $P_{n}$ is a unary predicate, interpreted in $\ct$ as the set of $n$-th powers. We let $T=\Th_{\cL}(\ct)$. It is folklore that
$T$ eliminates quantifiers in $\cL$ (\Cref{sec:theory}).

Our first main result characterizes the models of the universal part $\tf$ of $T$ among field extensions of a fixed model. The key tool is the \emph{$t$-Kochen operator}
\[
\gam(x)=\frac{x}{1+tx^{2}},
\]
a natural but not immediate analogue of $\gam_{p}$: since the residue field of $\ct$ is algebraically closed there is no Frobenius available, and $\gam$ operates instead through the parity obstruction ``$v(t)$ is odd''. For a field extension $L$ of $K\models T$ we let $\Lam_{L}=\big\{a/(1+tb):a,b\in \cO_{K}[\gam(L)]\big\}$ be the associated \emph{$t$-Kochen ring} (\cref{def:kochen}), where $\cO_K$ is the valuation ring of $K$. 

\begin{theoremA}[{\cref{thm:characterization}}]
Let $K\models T$ and let $L$ be a field extension of $K$. The following are equivalent:
\begin{enumerate}[label=\textnormal{(\arabic*)}]
\item $L$ is formally $T$-adic over $K$, that is, $L$ admits an $\cL$-structure extending that of $K$ which is a model of $\tf$;
\item there is a valuation $w$ on $L$ extending the valuation of $K$ whose value group $(\Gamma_{w},w(t))$ admits an expansion to a model of $\Presf$, the universal part of Presburger arithmetic;
\item $1/t\notin \cO_{K}[\gam(L)]$;
\item $1/t\notin \Lam_{L}$.
\end{enumerate}
\end{theoremA}

Our second main result is the corresponding Nullstellensatz, which we state in the language of $A$-radicals of \cite{CubidesNull}: for a subring $\KX\subseteq A\subseteq\KrX$ and an ideal $I=\langle f_{1},\dots,f_{l}\rangle$ of $\KX$, the \emph{$A$-radical} $\rt[A]{I}$ is the set of $g\in\KX$ such that $g^{m}=\sum_{i=1}^{l}\omega_{i}f_{i}$ for some $m\geqslant 1$ and some $\omega_{1},\dots,\omega_{l}\in A$. Here the relevant coefficient ring is $R=\Lam_{\KrX}\cdot\KX$.

\begin{theoremB}[{\cref{thm:nss}}]
Let $K\models T$ and let $I$ be an ideal of $\KX$. A polynomial $f\in\KX$ vanishes on $Z(I)$ if and only if $f\in\rt[R]{I}$.
\end{theoremB}

As in the classical situations, Theorem~B yields a dictionary between algebra and geometry: $R$-ideals (ideals with $I=\rt[R]{I}$) correspond bijectively to algebraic subsets of $K^{n}$, prime $R$-ideals to irreducible algebraic sets, and maximal $R$-ideals to points (\Cref{cor:bijections}). In the final part of the article we introduce the \emph{$T$-spectrum} of a $K$-algebra, in the spirit of the real spectrum and of the $p$-adic spectrum of Robinson, Bélair and Srhir \cite{Robinson86,Belair90,Srhir08}, and we compare its three natural incarnations --- prime ideals, morphisms into models of $T$, and types (\cref{sec:spectrum}).

Our proofs follow the strategy that Srhir \cite{srhir} employed for the $p$-adic Nullstellensatz, itself built on Jarden--Roquette \cite{jardenroquette}: an abstract Nullstellensatz for model complete theories of fields (\Cref{sec:abstract}), which reduces Theorem~B to the characterization of formally $T$-adic fields given by Theorem~A. The heart of the article is thus the proof of Theorem~A, which combines the $t$-Kochen operator with a valuation-theoretic analysis of value groups modeled on Presburger arithmetic. It has two main ingredients: \Cref{lem:gap} showing that the operator forces $w(t)$ to be the smallest positive element of the value group, and \Cref{lem:expansion} showing that this condition is exactly what is needed for the value group to model $\Presf$.

Let us mention that a uniform perspective on Nullstellensätze of this kind, in which the coefficient ring $R$ is replaced by a holomorphy ring of $K$-rational places over an arbitrary large field, is developed by the authors in \cite{CubidesNull}. The present article concentrates on the concrete description available for $\ct$ via the $t$-Kochen operator, which is finer than the holomorphy-ring description and is specific to this theory.

The article is organized as follows. \Cref{sec:prelim} collects preliminaries on valued fields, places and Presburger arithmetic. \Cref{sec:theory} gathers results about $\ct$ and its model theory. \Cref{sec:abstract} contains the abstract Nullstellensatz for model complete theories. \Cref{sec:formally} studies formally $T$-adic fields and proves the valuation-theoretic extension results. \Cref{sec:kochen} introduces the $t$-Kochen operator and proves Theorem~A. \Cref{sec:nss} proves Theorem~B and its corollaries. \Cref{sec:spectrum} develops the $T$-spectrum, and \cref{sec:questions} gathers open questions, notably concerning $\rlt$.

\subsection*{Acknowledgements and AI Disclaimer} All proofs here presented have been happily humanly produced. The use of AI has been restricted here to editorial corrections and proof-reading corrections, where it suggested modifications to shorten/improve presentation of some of the proofs. The project was financed by Universidad de los Andes FAPA grant (FAPA-599-22).  

\section{Preliminaries}\label{sec:prelim}

Throughout the article $K$ denotes a field of characteristic zero, $X=(X_{1},\dots,X_{n})$ a tuple of variables, $\KX=K[X_{1},\dots,X_{n}]$ the polynomial ring and $\KrX$ its fraction field. We assume familiarity with valued fields, places and basic model theory; most of our notation is standard and follows \cite{EP05,Endler72} for valuation theory and \cite{Marker02,Hodges93} for model theory. This section fixes notation and isolates the few less standard notions we will need.

\subsection{Valued fields and places}\label{subsec:valued}

For a valuation $v$ on a field $F$ we write $\cO_{v}$, $\cM_{v}$, $k_{v}$, $\Gamma_{v}$ and $\res\colon\cO_{v}\to k_{v}$ for its valuation ring, maximal ideal, residue field, value group and residue map. An extension of valued fields $(K,v)\subseteq(L,w)$ is a field extension with $\cO_{v}=\cO_{w}\cap K$; then $\Gamma_{v}\leqslant\Gamma_{w}$ and $k_{v}\subseteq k_{w}$ canonically. We let $K^{h}$ denote the henselization of $(K,v)$.

For a field $F$, we equip $F_{\infty}=F\cup\{\infty\}$ with the usual partial arithmetic: $a+\infty=\infty$ for $a\in F$; $a\cdot\infty=\infty$ for $a\in F_{\infty}\setminus\{0\}$; $0^{-1}=\infty$, $\infty^{-1}=0$, $-\infty=\infty$; while $\infty+\infty$ and $0\cdot\infty$ are undefined. A \emph{place} $\varphi\colon F_{\infty}\to K_{\infty}$ is a map with $\varphi(1)=1$ which preserves sums and products whenever the corresponding operations on the values are defined; see \cite[\S2.2]{EP05} or \cite[\S8]{Endler72}.

We will freely use the basic properties of places \cite[\S8]{Endler72}: the preimage $\varphi^{-1}(K)$ is a valuation ring of $F$ with maximal ideal $\varphi^{-1}(0)$; for $a\in F^{\times}$, $\varphi(a)=\infty$ if and only if $\varphi(a^{-1})=0$; and the composition of places is a place.

Given $f\in\KX$, $b\in K$ and $1\leqslant l\leqslant n$, write $\ev_{b,l}(f)=f(X_{1},\dots,X_{l-1},b,X_{l+1},\dots,X_{n})$. Two polynomials $f,g$ are \emph{$l$-coprime} if they are coprime as polynomials in $X_{l}$ over the ring $K[X_{1},\dots,\widehat{X_{l}},\dots,X_{n}]$, where $\widehat{X_{l}}$ means the variable $X_l$ is omitted. 

\begin{definition}\label{def:eval-places}
Let $b\in K$ and $1\leqslant l\leqslant n$. The \emph{$l$-evaluation place at $b$} is the map
\[
P_{b,l}\colon \KrX_{\infty}\longrightarrow K(X_{1},\dots,\widehat{X_{l}},\dots,X_{n})_{\infty},\qquad
P_{b,l}\Big(\frac{f}{g}\Big)=
\begin{cases}
\dfrac{\ev_{b,l}(f)}{\ev_{b,l}(g)} & \text{if } \ev_{b,l}(g)\neq 0,\\[4pt]
\infty & \text{otherwise,}
\end{cases}
\]
where $f,g$ are $l$-coprime. For $a=(a_{1},\dots,a_{n})\in K^{n}$ and a permutation $\sigma$ of $\{1,\dots,n\}$, the \emph{evaluation place associated to $(a,\sigma)$} is the composition
\[
P_{a,\sigma}=P_{a_{\sigma(n)},\sigma(n)}\circ\cdots\circ P_{a_{\sigma(1)},\sigma(1)}\colon \KrX_{\infty}\longrightarrow K_{\infty}.
\]
\end{definition}

Each $P_{b,l}$ is a place, hence so is each $P_{a,\sigma}$. Moreover $P_{a,\sigma}$ restricts to the identity on $K$, and for every polynomial $f\in\KX$ one has $P_{a,\sigma}(f)=f(a)$.

\begin{remark}\label{rem:permutations}
The permutation $\sigma$ in \cref{def:eval-places} genuinely matters: for $h=X_{1}/X_{2}$ and $a=(0,0)$, evaluating $X_{1}$ first yields $P_{a,\mathrm{id}}(h)=0$, whereas evaluating $X_{2}$ first yields $\infty$. Distinct pairs $(a,\sigma)$ may thus give distinct places with different associated valuation rings $\cO_{P_{a,\sigma}}$.
\end{remark}

\subsection{Radicals and zero sets}\label{subsec:radicals}

Inspired by the $p$-adic Nullstellensatz, we will use the following relative notion of radical.

\begin{definition}[{\cite{CubidesNull}}]\label{def:A-radical}
Let $A$ be a subring of $\KrX$ containing $\KX$ and let $I=\langle f_{1},\dots,f_{l}\rangle$ be an ideal of $\KX$. The \emph{$A$-radical} of $I$ is
\[
\rt[A]{I}=\Big\{g\in\KX \;:\; g^{m}=\sum_{i=1}^{l}\omega_{i}f_{i}\ \text{ for some } m\geqslant1 \text{ and } \omega_{1},\dots,\omega_{l}\in A\Big\}.
\]
\end{definition}

\begin{remark}\label{rem:A-radical}
Equivalently, $\rt[A]{I}=\KX\cap\sqrt{IA}$, where $IA$ denotes the ideal of $A$ generated by $I$ and $\sqrt{IA}$ its radical in $A$. In particular $\rt[A]{I}$ is a radical ideal of $\KX$ containing $I$, and it does not depend on the chosen generators. Note that $\rt[{\KX}]{I}=\sqrt{I}$ recovers the classical radical.
\end{remark}

Recall also that in a noetherian ring every radical ideal $I$ is the intersection $I=P_{1}\cap\dots\cap P_{m}$ of the finitely many primes minimal over $I$, and $P_{i}\not\subseteq P_{j}$ for $i\neq j$; see \cite[Chapters 4 and 7]{AM69}.

\begin{definition}\label{def:notation-algebraic}
For $S\subseteq\KX$ and $B\subseteq K^{n}$ set
\[
Z(S)=\{a\in K^{n}: f(a)=0\text{ for all }f\in S\},\qquad
J(B)=\{f\in\KX: f(a)=0\text{ for all }a\in B\}.
\]
A subset of $K^{n}$ of the form $Z(I)$ is an \emph{algebraic set}; an ideal $I$ with $J(Z(I))=I$ is called \emph{closed}. 
\end{definition}

\begin{fact}\label{lem:JV-props}
Let $I$ be an ideal of $\KX$. Then:
\begin{enumerate}[label=\textnormal{(\arabic*)}]
\item $J(Z(I))$ is a radical closed ideal;
\item every prime ideal minimal over $J(Z(I))$ is closed.
\end{enumerate}
\end{fact}

These facts are well known: (1) is formal, and (2) corresponds, via the correspondence between closed ideals and algebraic sets, to the decomposition of an algebraic set into irreducible components; see e.g.\ \cite[Proposition 2.8.2]{bcr} where the result is stated there over a real closed field but remains valid, with the same proof, over an arbitrary field.

\subsection{Model-theoretic background}\label{subsec:model-theory}

Let $\cL_{0}$ be a first-order language, $T_{0}$ an $\cL_{0}$-theory and $\mathbf{M}$ an $\cL_{0}$-structure. We write $(T_{0})_\forall$ for the set of universal $\cL_{0}$-sentences implied by $T_{0}$, and $\Delta(\mathbf{M})$ for the atomic diagram of $\mathbf{M}$. 


\begin{fact}\label{lem:mc-maps}
Let $T_{0}$ be a model complete theory in a language containing the language of rings, all of whose models are fields. Let $\mathbf{K},\mathbf{L}\models T_{0}$ with underlying fields $K$ and $L$, suppose $K\subseteq L$ is an $\cL_0$-substructure, and let $I$ be a proper ideal of $\KX$. If there is a $K$-algebra homomorphism $\varphi\colon\KX/I\to L$, then there is a $K$-algebra homomorphism $\psi\colon\KX/I\to K$.
\end{fact}


Let $\cL_{\Pres}=\{+,-,<,0,1\}\cup\{D_{n}:n\geqslant2\}$, where each $D_{n}$ is a unary predicate, and let $\Pres$ (\emph{Presburger arithmetic}) be the theory of $\mathbf{Z}=(\Z,+,-,<,0,1,(n\Z)_{n})$ in this language. The theory $\Pres$ eliminates quantifiers and admits a well-known recursive axiomatization \cite[\S3.1]{Marker02}; its models are called \emph{$\Z$-groups}. 
We write $\Presf$ for the universal part of $\Pres$.
We will also use the following refinement \cite[Lemmas 3.1.19 and 3.1.20]{Marker02}: every $\mathbf{G}=(G,+,-,<,0,1,(D_{n})_{n})\models\Presf$ embeds into a \emph{minimal} $\Z$-group, namely the subgroup
\[
H=\big\{g/n \;:\; g\in G,\ n\geqslant1,\ \mathbf{G}\models D_{n}(g)\big\}
\]
of the divisible hull of $G$, which moreover satisfies $nH\cap G=D_{n}^{\mathbf{G}}$ for all $n\geqslant 2$.

\section{The theory of \texorpdfstring{$\ct$}{C((t))}}\label{sec:theory}

Recall that for a field $k$ the Laurent series field $k(\!(t)\!)$ is henselian for the $t$-adic valuation $v(\sum_{i} a_{i}t^{i})=\min\{i:a_{i}\neq0\}$, with value group $\Z$ and residue field $k$ \cite[Chapter 4]{EP05}. In particular $\ct$ is a henselian valued field of equicharacteristic $0$ with algebraically closed residue field, and $v(t)=1$ is the smallest positive element of the value group.

We work in the one-sorted languages
\[
\cL_{t}=\{+,-,\cdot,0,1,t\},\qquad
\cL=\cL_{t}\cup\{{}^{-1}\}\cup\{P_{n}:n\geqslant2\},
\]
where $t$ is a constant symbol, ${}^{-1}$ a unary function symbol (with the convention $0^{-1}=0$) and each $P_{n}$ a unary predicate, interpreted in $\ct$ as the set of $n$-th powers. We set
\[
T=\Th_{\cL}(\ct),\qquad T_{t}=\Th_{\cL_{t}}(\ct).
\]
The starting point is the folklore observation that the valuation ring of $\ct$ is existentially definable in $\cL_{t}$; we recall it without proof (both directions are immediate from Hensel's lemma together with the fact that minimality of $w(t)$ forces $w(t)\notin2\Gamma_{w}$; compare \cite[p.~286]{Robinson86} for the $p$-adic analogue).

\begin{fact}\label{prop:def-valuation}
Let $(L,w)$ be a henselian valued field with $\operatorname{char}(k_{w})\neq 2$, let $t\in L^{\times}$ and suppose that $w(t)$ is the smallest positive element of $\Gamma_{w}$. Then for every $x\in L$,
\[
w(x)\geqslant 0\quad\Longleftrightarrow\quad L\models\exists y\,(y^{2}=1+tx^{2}).
\]
In particular this holds for $L=\ct$.
\end{fact}

Since the valuation of a model of $T$ is thus definable by the $\cL_{t}$-formula $\varphi_{v}(x):=\exists y\,(y^{2}=1+tx^{2})$, we will abuse notation and freely use the symbol $v$ inside $\cL_{t}$- and $\cL$-formulas, writing $v(x)\geqslant0$, $v(x)=0$, $v(x)\geqslant v(y)$, $v(x)>v(y)$, etc., for the corresponding Boolean combinations of instances of $\varphi_{v}$ (for example, $v(x)\geqslant v(y)$ stands for $y\neq0\wedge\varphi_{v}(x/y)$).

\begin{lemma}\label{lem:powers}
Let $(L,w)$ be a henselian valued field of equicharacteristic $0$ whose residue field is algebraically closed, and let $t\in L$ be such that $w(t)$ is the smallest positive element of $\Gamma_{w}$. Then for every $x\in L^{\times}$ and every $n\geqslant2$,
\[
x\in(L^{\times})^{n}\iff w(x)\in n\Gamma_{w}.
\]
If moreover $\Gamma_{w}$ is a $\Z$-group with $1=w(t)$, then for each $x\in L^{\times}$ and $n\geqslant 2$ there is exactly one $i\in\{0,\dots,n-1\}$ such that $xt^{i}\in(L^{\times})^{n}$.
\end{lemma}

\begin{proof}
If $x=y^{n}$ then $w(x)=nw(y)\in n\Gamma_{w}$. Conversely, if $w(x)=nw(y)$ for some $y\in L^{\times}$, then $u=x/y^{n}$ is a unit of $\cO_{w}$. Since $k_{w}$ is algebraically closed of characteristic $0$, the reduction of $Y^{n}-u$ has a simple root: any root $\alpha$ of $Y^{n}-\res(u)$ is nonzero and the derivative $n\alpha^{n-1}$ does not vanish. By henselianity $u\in(L^{\times})^{n}$, hence $x=(zy)^{n}$ for suitable $z$. The last claim follows: $xt^{i}\in(L^{\times})^{n}$ if and only if $w(x)+i\,w(t)\in n\Gamma_{w}$, and in a $\Z$-group every element is congruent modulo $n\Gamma_{w}$ to $-i\cdot 1$ for exactly one $i\in\{0,\dots,n-1\}$.
\end{proof}

The following gathers the model-theoretic properties of $T$ we shall use. The proof is a standard consequence of Ax-Kochen-Ershov and well established quantifier elimination results in valued fields. See for example \cite{vdD14}.

\begin{fact}\label{thm:qe}\
\begin{enumerate}[label=\textnormal{(\arabic*)}]
\item An $\cL$-structure $\mathbf{L}$, with underlying field $L$, is a model of $T$ if and only if the following hold: the formula $\varphi_{v}$ defines in $L$ the valuation ring of a henselian valuation $w$ of equicharacteristic $0$; the residue field $k_{w}$ is algebraically closed; the value group $\Gamma_{w}$ is a $\Z$-group whose smallest positive element is $w(t)$; and $P_{n}^{\mathbf{L}}$ is the set of $n$-th powers of $L$ for every $n$. 
\item $T$ eliminates quantifiers in $\cL$.
\item $T_{t}$ is model complete.
\item Every $\cL_{t}$-embedding between models of $T_{t}$ is an elementary $\cL$-embedding.
\end{enumerate}
\end{fact}


\begin{remark}\label{rem:Tforall-language}
From now on we write $\tf$ for the universal part of $T$ in the language $\cL$. Since $\cL$ contains ${}^{-1}$, every substructure of a model of $T$ is a field; the models of $\tf$ are thus, up to isomorphism, exactly the $\cL$-substructures of models of $T$. Note also that in a model $\mathbf{L}\models\tf$ the predicate $P_{n}$ need \emph{not} coincide with the set of $n$-th powers of $L$: the axiom $P_{n}(x)\leftrightarrow\exists y\,(y^{n}=x)$ is not universal, and only the implication $\forall x\,P_{n}(x^{n})$ survives in $\tf$. 
\end{remark}

\section{An abstract Nullstellensatz for model complete theories}\label{sec:abstract}

In this section we isolate the model-theoretic mechanism common to all Nullstellensätze mentioned in the introduction; compare with Weispfenning's framework \cite{Weispfenning77}. The results in this section go back to ideas of Abraham Robinson. We include the proofs since they might be of interest. Throughout the section we fix:
\begin{itemize}
\item a language $\cL_{0}$ obtained from $\{+,-,\cdot,{}^{-1},0,1\}$ by adding constant and relation symbols;
\item a model complete $\cL_{0}$-theory $T_{0}$ containing the theory of fields;
\item a model $\mathbf{K}\models T_{0}$ with underlying field $K$.
\end{itemize}
Since $\cL_{0}$ has no function symbols beyond the field operations, the substructures of a model of $T_{0}$ are exactly the subfields containing the constants, with the induced relations; the models of $(T_{0})_{\forall}$ are thus exactly such subfields of models of $T_{0}$.

\begin{notation}
For $a\in K^{n}$ we write $I_{a}=J(\{a\})=\{f\in\KX: f(a)=0\}$, a maximal ideal with $\KX/I_{a}\cong K$. For a prime ideal $I\subseteq\KX$ we write $k(I)=\Frac(\KX/I)$ and $\bar{f}$ for the image of $f\in\KX$ in $k(I)$.
\end{notation}

\begin{lemma}\label{lem:fip}
Let $I\subseteq\KX$ be a prime closed ideal. Then the family of subsets of $K^{n}$
\[
\cF=\big\{Z(g)^{c} : g\in\KX,\ Z(g)^{c}\cap Z(I)\neq\emptyset\big\}\cup\{Z(I)\}
\]
has the finite intersection property.
\end{lemma}

\begin{proof}
First, $Z(I)\neq\emptyset$: otherwise $J(Z(I))=J(\emptyset)=\KX$, contradicting $J(Z(I))=I\subsetneq\KX$.

Let $Z(g_{1})^{c},\dots,Z(g_{m})^{c}\in\cF$ and suppose $Z(g_{1})^{c}\cap\dots\cap Z(g_{m})^{c}\cap Z(I)=\emptyset$. Setting $g=g_{1}\cdots g_{m}$, every $a\in Z(I)$ then lies in some $Z(g_{i})$, so $g(a)=0$; hence $g\in J(Z(I))=I$. As $I$ is prime, $g_{i}\in I$ for some $i$, which gives $Z(I)\subseteq Z(g_{i})$, i.e.\ $Z(g_{i})^{c}\cap Z(I)=\emptyset$, contradicting $Z(g_{i})^{c}\in\cF$.
\end{proof}

\begin{theorem}\label{thm:closed-implies-model}
Let $I\subseteq\KX$ be a prime closed ideal. Then $k(I)$ can be equipped with an $\cL_{0}$-structure $\mathbf{k(I)}$ such that $\mathbf{k(I)}\models(T_{0})_{\forall}$ and the canonical embedding $K\to k(I)$ is an $\cL_{0}$-embedding.
\end{theorem}

\begin{proof}
By \cref{lem:fip} there is an ultrafilter $\cU$ on $K^{n}$ containing $\cF$. Let $\mathbf{L}=\mathbf{K}^{K^{n}}/\,\cU$ be the corresponding ultrapower, with underlying field $L$, and let
\[
\rho\colon\KX\longrightarrow L,\qquad \rho(f)=\big(f(a)\big)_{a\in K^{n}}/\,\cU .
\]
This is a ring homomorphism whose restriction to $K$ is the diagonal embedding $K\preceq L$, which is elementary by \L o\'s's theorem; in particular $\mathbf{L}\models T_{0}$ and the diagonal is an $\cL_{0}$-embedding.

We claim $\ker\rho=I$. If $f\in I$ then $Z(I)\subseteq Z(f)$, and $Z(I)\in\cF\subseteq\cU$, so $Z(f)\in\cU$ and $\rho(f)=0$. Conversely, if $f\notin I$ then, since $J(Z(I))=I$, there is $a\in Z(I)$ with $f(a)\neq 0$, so $Z(f)^{c}\cap Z(I)\neq\emptyset$ and hence $Z(f)^{c}\in\cF\subseteq\cU$; thus $Z(f)\notin\cU$ and $\rho(f)\neq 0$.

Consequently $\rho$ induces a field embedding $k(I)=\Frac(\KX/I)\hookrightarrow L$ extending the diagonal on $K$. Identify $k(I)$ with its image. Since the constants of $\cL_{0}$ are interpreted in $K\subseteq k(I)$, the subfield $k(I)$ is an $\cL_{0}$-substructure of $\mathbf{L}$ for the induced relations. Being a substructure of a model of $T_{0}$, it is a model of $(T_{0})_{\forall}$, and $K\to k(I)$ is an $\cL_{0}$-embedding because both structures are induced from $\mathbf{L}$.
\end{proof}

The following theorem can be seen as an abstract form of Nullstellensatz for model complete theories of fields (under the assumptions of $\cL_0$). 

\begin{theorem}\label{thm:model-implies-closed}
Let $I\subseteq\KX$ be a prime ideal and suppose $k(I)$ carries an $\cL_{0}$-structure $\mathbf{k(I)}\models(T_{0})_{\forall}$ such that the canonical embedding $K\to k(I)$ is an $\cL_{0}$-embedding. Then $I$ is closed, i.e.\ $J(Z(I))=I$.
\end{theorem}

\begin{proof}
Since $\mathbf{k(I)}\models(T_{0})_{\forall}$, there are $\mathbf{M}\models T_{0}$, with underlying field $M$, and an $\cL_{0}$-embedding $\mathbf{k(I)}\to\mathbf{M}$. The composite $\mathbf{K}\to\mathbf{k(I)}\to\mathbf{M}$ is an $\cL_{0}$-embedding between models of $T_{0}$, hence elementary by model completeness. Let $\varphi\colon\KX\to M$ be the $K$-algebra map sending $X_{i}$ to the image of $\overline{X_{i}}$; by construction $\ker\varphi=I$.

Fix generators $f_{1},\dots,f_{l}$ of $I$ and let $f\in J(Z(I))$. The $\cL_{0}$-sentence with parameters in $K$
\[
\forall x_{1}\dots\forall x_{n}\Big(\bigwedge_{i=1}^{l}f_{i}(x_{1},\dots,x_{n})=0\ \to\ f(x_{1},\dots,x_{n})=0\Big)
\]
holds in $\mathbf{K}$, hence in $\mathbf{M}$ by elementarity. Applying it to the tuple $\big(\varphi(X_{1}),\dots,\varphi(X_{n})\big)$, which annihilates every $f_{i}$, we get $\varphi(f)=0$, i.e.\ $f\in I$. The reverse inclusion $I\subseteq J(Z(I))$ is trivial.
\end{proof}

\begin{theorem}\label{thm:bijections-abstract}
The operators $Z$ and $J$ induce mutually inverse bijections:
\begin{enumerate}[label=\textnormal{(\arabic*)}]
\item between closed ideals of $\KX$ and algebraic subsets of $K^{n}$;
\item between prime closed ideals of $\KX$ and irreducible algebraic subsets of $K^{n}$;
\item between maximal proper closed ideals of $\KX$ and points of $K^{n}$.
\end{enumerate}
\end{theorem}

\begin{proof}
Parts (1) and (2) are the standard dictionary between closed ideals and algebraic sets and hold, with the classical proofs, over an arbitrary field; compare \cite[Proposition 2.8.2]{bcr}. Only (3) uses model completeness. Each $I_{a}$ is a maximal ideal of $\KX$ and it is closed since $I_{a}\subseteq J(Z(I_{a}))\subseteq J(\{a\})=I_{a}$. Conversely, let $I$ be a maximal element of the set of proper closed ideals. Since $I$ is radical (\cref{lem:JV-props}(1)) and its minimal primes are closed (\cref{lem:JV-props}(2)), maximality gives that $I$ is prime. By \cref{thm:closed-implies-model}, $k(I)$ carries a structure $\mathbf{k(I)}\models(T_{0})_{\forall}$, which embeds into some $\mathbf{M}\models T_{0}$ over $\mathbf{K}$; the field $M$ then contains $K$ as an $\cL_{0}$-substructure. \Cref{lem:mc-maps} applied to the composite $\KX/I\to M$ yields a $K$-algebra homomorphism $\psi\colon\KX/I\to K$. Let $a_{i}=\psi(\overline{X_{i}})$ and $a=(a_{1},\dots,a_{n})$. For every $f\in I$ we get $f(a)=\psi(\bar f)=0$, so $a\in Z(I)$ and thus $I\subseteq J(\{a\})=I_{a}$. By maximality $I=I_{a}$.
\end{proof}

\section{Formally \texorpdfstring{$T$}{T}-adic fields}\label{sec:formally}

For the rest of the article we fix a model $\mathbf{K}\models T$ with underlying field $K$, valuation $v$, valuation ring $\cO_{K}$, residue field $k_{K}$ and value group $\Gamma_{K}$; recall from \cref{thm:qe}(1) that $\Gamma_{K}$ is a $\Z$-group with smallest positive element $v(t)$, and from \cref{lem:powers} that $P_{n}^{\mathbf{K}}=\{x\in K^{\times}: v(x)\in n\Gamma_{K}\}\cup\{0\}$. All field extensions $L$ of $K$ are assumed to interpret the constant $t$ as the element $t\in K$.

\begin{definition}\label{def:formally-T-adic}
Let $L$ be a field extension of $K$.
\begin{enumerate}[label=\textnormal{(\arabic*)}]
\item $L$ is \emph{formally $T$-adic} (over $K$) if $L$ can be equipped with interpretations of the predicates $P_{n}$ making it a model of $\tf$ in such a way that the inclusion $\mathbf{K}\subseteq\mathbf{L}$ is an $\cL$-embedding.
\item A valuation $w$ on $L$ extending $v$ is a \emph{$\tf$-valuation} if the pair $(\Gamma_{w},w(t))$ can be expanded by predicates $(D_{n})_{n\geqslant 2}$ to an $\cL_{\Pres}$-structure which is a model of $\Presf$, the constant $1$ being interpreted as $w(t)$.
\end{enumerate}
\end{definition}

Note that if $w$ is a $\tf$-valuation then $w(t)$ is the smallest positive element of $\Gamma_{w}$, since the universal sentences $0<1$ and $\forall x\,(x\leqslant 0\vee x\geqslant 1)$ belong to $\Presf$.

The main result of this section is that clauses (1) and (2) of \Cref{def:formally-T-adic} are equivalent conditions on $L$. The implication from (1) to (2) is a straightforward restriction argument. For the converse we must extend $(L,w)$ to a model of $T$, which requires enlarging the value group to a $\Z$-group and the residue field to an algebraically closed field without disturbing what is below. We first prove the required extension theorem.

\begin{lemma}\label{lem:division-step}
Let $(L,w)$ be a valued field, $p$ a prime number and $\delta\in\Gamma_{w}$ with $\delta/p\notin\Gamma_{w}$ (inside the divisible hull of $\Gamma_{w}$). Then there is an algebraic extension $(L',w')$ of $(L,w)$ with
\[
[L':L]=p,\qquad \Gamma_{w'}=\Gamma_{w}+\Z\tfrac{\delta}{p},\qquad k_{w'}=k_{w}.
\]
\end{lemma}

\begin{proof}
Pick $a\in L$ with $w(a)=\delta$. The polynomial $Y^{p}-a$ is irreducible over $L$: since $p$ is prime, $Y^{p}-a$ is reducible if and only if $a\in L^{p}$ \cite[Chapter VI, Theorem 9.1]{lang-algebra}, and $a=c^{p}$ would give $\delta/p=w(c)\in\Gamma_{w}$. Let $b$ be a root of $Y^{p}-a$ and $L'=L(b)$, so $[L':L]=p$. The valuation $w$ extends to a valuation $w'$ on $L'$ with $\Gamma_{w'}=\Gamma_{w}+\Z\frac{\delta}{p}$: explicitly, $w'\big(\sum_{i=0}^{p-1}a_{i}b^{i}\big)=\min_{i}\big\{w(a_{i})+i\,\frac{\delta}{p}\big\}$ defines such an extension \cite[Remark 3.2.17]{EP05}. The quotient $\Gamma_{w'}/\Gamma_{w}$ is generated by $\frac{\delta}{p}+\Gamma_{w}$, an element of order exactly $p$ (its order divides the prime $p$ and is not $1$), so $[\Gamma_{w'}:\Gamma_{w}]=p$. By the fundamental inequality \cite[Theorem 3.3.4]{EP05},
\[
p\cdot[k_{w'}:k_{w}]=[\Gamma_{w'}:\Gamma_{w}]\cdot[k_{w'}:k_{w}]\leqslant[L':L]=p,
\]
hence $k_{w'}=k_{w}$.
\end{proof}

\begin{theorem}\label{thm:extension}
Let $(L,w)$ be a valued field of equicharacteristic $0$, let $\varepsilon\in\Gamma_{w}$, and suppose $(\Gamma_{w},\varepsilon)$ admits an expansion $(\Gamma_{w},\varepsilon,(D_{n})_{n})\models\Presf$. Let $k'$ be an algebraic extension of $k_{w}$. Then there is an algebraic extension $(F,u)$ of $(L,w)$ such that:
\begin{enumerate}[label=\textnormal{(\arabic*)}]
\item $(F,u)$ is henselian with residue field $k'$;
\item $\Gamma_{u}=H$, the minimal $\Z$-group containing $(\Gamma_{w},\varepsilon,(D_{n})_{n})$ (see \Cref{subsec:model-theory}); in particular $\Gamma_{u}$ is a $\Z$-group with smallest positive element $\varepsilon$ and $n\Gamma_{u}\cap\Gamma_{w}=D_{n}$ for all $n$.
\end{enumerate}
\end{theorem}

\begin{proof}
Replacing $(L,w)$ by its henselization \cite[Theorem 5.2.2]{EP05}, which is an algebraic extension with the same value group and residue field, we may assume $(L,w)$ henselian. Fix an algebraic closure $L^{\mathrm{alg}}$ of $L$; being henselian, $(L,w)$ admits a unique extension of its valuation to every subextension of $L^{\mathrm{alg}}/L$, and we denote all these extensions again by $w$. The value group of any algebraic extension of $(L,w)$ lies inside the divisible hull of $\Gamma_{w}$, which also contains $H$.

Consider the family
\[
\mathcal{Z}=\big\{L'\,:\, L\subseteq L'\subseteq L^{\mathrm{alg}},\ \Gamma_{w|_{L'}}\subseteq H,\ k_{w|_{L'}}=k_{w}\big\},
\]
partially ordered by inclusion. It contains $L$, and it is closed under unions of chains, since the value group (resp.\ residue field) of a union is the union of the value groups (resp.\ residue fields). By Zorn's lemma, $\mathcal{Z}$ has a maximal element $\widetilde{F}$; write $\widetilde{\Gamma}$ for its value group.

We claim $\widetilde{\Gamma}=H$. Suppose not, and pick $h\in H\setminus\widetilde{\Gamma}$. Since $h\in H$ we have $mh\in\Gamma_{w}\subseteq\widetilde{\Gamma}$ for some $m\geqslant 1$, so there is a minimal $m\geqslant 2$ with $mh\in\widetilde{\Gamma}$. Let $p$ be a prime divisor of $m$ and set $\delta'=(m/p)\,h$. By minimality of $m$, $\delta'\notin\widetilde{\Gamma}$, while $\delta:=p\,\delta'=mh\in\widetilde{\Gamma}$. \Cref{lem:division-step} applied to $(\widetilde{F},w)$ yields an algebraic extension $\widetilde{F}'\subseteq L^{\mathrm{alg}}$ (the uniqueness of valuation extensions ensuring compatibility with $w$) with value group $\widetilde{\Gamma}+\Z\delta'\subseteq H$ and unchanged residue field. Thus $\widetilde{F}'\in\mathcal{Z}$ strictly contains $\widetilde{F}$, contradicting maximality. Hence $\Gamma_{w|_{\widetilde F}}=H$ and $k_{w|_{\widetilde F}}=k_{w}$.

Finally, $\widetilde{F}$ is henselian (algebraic extensions of henselian fields are henselian) and of equicharacteristic $0$, so there is an algebraic extension $(F,u)$ of $(\widetilde{F},w)$ with residue field $k'$ and the same value group $H$ \cite[Theorem 5.2.7]{EP05}; it is henselian, being algebraic over a henselian field. Property (2) holds by the properties of $H$ recalled in \cref{subsec:model-theory}: $H$ is a $\Z$-group whose distinguished element $1$ is $\varepsilon$, so $\varepsilon$ is its smallest positive element, and $nH\cap\Gamma_{w}=D_{n}$.
\end{proof}

Before stating the equivalence we isolate a purity property of extensions of $\Z$-groups which will give the compatibility over $K$.

\begin{lemma}\label{lem:purity}
Let $G\subseteq H$ be $\Z$-groups with the same smallest positive element $1$. Then $nH\cap G=nG$ for every $n\geqslant 1$.
\end{lemma}

\begin{proof}
Only $nH\cap G\subseteq nG$ needs proof. Let $g\in G$ with $g=nh$, $h\in H$. Since $G$ is a $\Z$-group, there is a (unique) $r\in\{0,\dots,n-1\}$ with $g-r\cdot 1\in nG$. Then $r\cdot 1=g-(g-r\cdot1)\in nH$. But in the $\Z$-group $H$ there is exactly one $i\in\{0,\dots,n-1\}$ with $0-i\cdot 1\in nH$, namely $i=0$; applying this to $-r$ modulo $n$, we conclude $r=0$ and $g\in nG$.
\end{proof}

\begin{proposition}\label{prop:link}
Let $L$ be a field extension of $K$. The following are equivalent:
\begin{enumerate}[label=\textnormal{(\arabic*)}]
\item $L$ is formally $T$-adic over $K$;
\item $L$ admits a $\tf$-valuation.
\end{enumerate}
\end{proposition}

\begin{proof}
(1)$\Rightarrow$(2). Let $\mathbf{L}\models\tf$ be an $\cL$-structure on $L$ with $\mathbf{K}\subseteq\mathbf{L}$ an $\cL$-embedding. There is an $\cL$-embedding $\mathbf{L}\to\mathbf{M}$ with $\mathbf{M}\models T$; we may assume $\mathbf{K}\subseteq\mathbf{L}\subseteq\mathbf{M}$ as $\cL$-substructures. Then $\mathbf{K}\subseteq\mathbf{M}$ is an $\cL$-embedding between models of $T$, hence elementary by \cref{thm:qe}. Let $v_{M}$ be the valuation of $\mathbf{M}$ and define $w$ on $L$ by $\cO_{w}=\cO_{v_{M}}\cap L$. For $x\in K$ we have $x\in\cO_{v_{M}}$ iff $\mathbf{M}\models\varphi_{v}(x)$ iff $\mathbf{K}\models\varphi_{v}(x)$ iff $x\in\cO_{K}$, using elementarity; hence $\cO_{w}\cap K=\cO_{K}$, i.e.\ $w$ extends $v$. Moreover, setting $D_{n}=n\Gamma_{M}\cap\Gamma_{w}$, the structure $(\Gamma_{w},w(t),(D_{n})_{n})$ is an $\cL_{\Pres}$-substructure of the $\Z$-group $(\Gamma_{M},v_M(t),(n\Gamma_{M})_{n})$, hence a model of the universal theory $\Presf$. Thus $w$ is a $\tf$-valuation.

(2)$\Rightarrow$(1). Let $w$ be a $\tf$-valuation on $L$, with expansion $(\Gamma_{w},w(t),(D_{n})_{n})\models\Presf$. Apply \cref{thm:extension} with $\varepsilon=w(t)$ and $k'$ an algebraic closure of $k_{w}$: we obtain an algebraic extension $(F,u)$ of $(L,w)$, henselian of equicharacteristic $0$, with algebraically closed residue field and value group a $\Z$-group with smallest positive element $u(t)=w(t)$. By \cref{thm:qe}(1), interpreting each $P_{n}$ as the set of $n$-th powers makes $F$ a model of $T$; call this structure $\mathbf{F}$.

Restrict the $\cL$-structure of $\mathbf{F}$ to the subfield $L$: since $\cL$ adds only relation and constant symbols to the language of fields, and $t\in K\subseteq L$, this makes $L$ an $\cL$-substructure $\mathbf{L}\subseteq\mathbf{F}$, so $\mathbf{L}\models\tf$.

It remains to check that $\mathbf{K}\subseteq\mathbf{L}$ is an $\cL$-embedding, i.e.\ that $P_{n}^{\mathbf{F}}\cap K=P_{n}^{\mathbf{K}}$ for all $n$. Both $\Gamma_{K}$ and $\Gamma_{u}$ are $\Z$-groups with smallest positive element $v(t)=u(t)$, and $\Gamma_{K}\subseteq\Gamma_{u}$ since $u$ extends $v$. By \cref{lem:powers} applied in $\mathbf{F}$ and in $\mathbf{K}$, and by \cref{lem:purity},
\[
x\in P_{n}^{\mathbf{F}}\cap K^\times \iff u(x)\in n\Gamma_{u}\cap\Gamma_{K}=n\Gamma_{K}\iff x\in P_{n}^{\mathbf{K}},
\]
as required.
\end{proof}

\begin{corollary}\label{cor:gauss-T-adic}
The Gauss valuation on $\KrX$ extending $v$ is a $\tf$-valuation. In particular the rational function field $\KrX$ is formally $T$-adic over $K$.
\end{corollary}

\begin{proof}
Recall that the Gauss valuation on $K(X_{1})$ is defined by $w\big(\sum_{\nu}a_{\nu}X_{1}^{\nu}\big)=\min_{\nu}v(a_{\nu})$ on polynomials and extended to quotients; iterating coordinate by coordinate yields a valuation $w$ on $\KrX$ which extends $v$ with $\Gamma_{w}=\Gamma_{K}$ and $w(t)=v(t)$ \cite[Corollary 2.2.2]{EP05}. The expansion $(\Gamma_{K},v(t),(n\Gamma_{K})_{n})$ is a $\Z$-group (\cref{thm:qe}(1)), in particular a model of $\Presf$. Conclude by \cref{prop:link}.
\end{proof}

\begin{remark}\label{rem:no-residue}
Being formally $T$-adic imposes no restriction on residue fields: only the value group is constrained. Indeed, for \emph{any} field extension $k_{0}$ of $\C$, the Laurent series field $L=k_{0}(\!(t)\!)\supseteq\ct$ has value group $\Z$ with $w(t)=1$ minimal, so $L$ is formally $T$-adic over $\ct$ by \cref{prop:link}, with residue field $k_{0}$. 
\end{remark}

\section{The \texorpdfstring{$t$}{t}-Kochen operator and the characterization theorem}\label{sec:kochen}

\begin{definition}\label{def:kochen}
Let $L$ be a field extension of $K$.
\begin{enumerate}[label=\textnormal{(\arabic*)}]
\item The \emph{$t$-Kochen operator} is the rational map $\gam(x)=\dfrac{x}{1+tx^{2}}$, defined at all $x\in L$ with $1+tx^{2}\neq0$.
\item $\gam(L)=\{\gam(x): x\in L,\ 1+tx^{2}\neq 0\}$.
\item The \emph{$t$-Kochen ring} of $L$ is
\[
\Lam_{L}=\Big\{\frac{a}{1+tb}\;:\;a,b\in\cO_{K}[\gam(L)],\ 1+tb\neq0\Big\}.
\]
\end{enumerate}
\end{definition}

It is not difficult to see that $\Lam_{L}$ is a subring of $L$ containing $\cO_{K}[\gam(L)]$. 


\begin{proposition}\label{prop:gamma-into-O}
Let $w$ be a valuation on $L$ such that $w(t)$ is the smallest positive element of $\Gamma_{w}$. Then:
\begin{enumerate}[label=\textnormal{(\arabic*)}]
\item $1+tx^{2}\neq0$ for every $x\in L$; in particular $\gam\colon L\to L$ is total;
\item $w(\gam(x))\geqslant 0$ for every $x\in L$;
\item if moreover $w$ extends $v$, then $\Lam_{L}\subseteq\cO_{w}$ and $w(\lambda)=w(a)$ for $\lambda=a/(1+tb)$ as in \cref{def:kochen}(3).
\end{enumerate}
\end{proposition}

\begin{proof}
Note first that minimality of $w(t)$ implies $w(t)\notin2\Gamma_{w}$: if $w(t)=2\gamma$ then $0<\gamma<w(t)$, a contradiction.

(1) If $1+tx^{2}=0$ then $w(t)+2w(x)=w(-1)=0$, so $w(t)=-2w(x)\in2\Gamma_{w}$, a contradiction.

(2) If $w(x)\geqslant0$ then $w(tx^{2})\geqslant w(t)>0$, so $w(1+tx^{2})=0$ and $w(\gam(x))=w(x)\geqslant0$. If $w(x)<0$ then $w(x)\leqslant -w(t)$ by minimality, so $w(tx^{2})=w(t)+2w(x)\leqslant-w(t)<0$; hence $w(1+tx^{2})=w(t)+2w(x)$ and
\[
w(\gam(x))=w(x)-w(t)-2w(x)=-w(x)-w(t)\geqslant0.
\]

(3) By (2) and $\cO_{K}\subseteq\cO_{w}$ (as $w$ extends $v$) we get $\cO_{K}[\gam(L)]\subseteq\cO_{w}$. For $b\in\cO_{K}[\gam(L)]$ we then have $w(tb)\geqslant w(t)>0$, so $w(1+tb)=0$, and $w(a/(1+tb))=w(a)\geqslant0$.
\end{proof}

The next result shows that over the base model the Kochen operator generates the whole valuation ring, in analogy with \cite[Theorem 6.4]{prestelroquette}.

\begin{proposition}\label{prop:gamma-K}
$\gam(K)=\cO_{K}$ and $\Lam_{K}=\cO_{K}[\gam(K)]=\cO_{K}$.
\end{proposition}

\begin{proof}
The inclusion $\gam(K)\subseteq\cO_{K}$ holds by \cref{prop:gamma-into-O}(2) applied to $w=v$. Conversely, let $a\in\cO_{K}$ and consider $f(Y)=atY^{2}-Y+a\in\cO_{K}[Y]$. Its reduction is $-Y+\res(a)$, which has the simple root $\res(a)$ (the derivative reduces to $-1$). Since $(K,v)$ is henselian, there is $b\in\cO_{K}$ with $f(b)=0$ and hence $a(1+tb^{2})=b$; note $1+tb^{2}\ne0$ by \cref{prop:gamma-into-O}(1), so $a=\gam(b)$. Thus $\cO_{K}=\cO_{K}[\gam(K)]$, and then every element $a/(1+tb)$ of $\Lam_{K}$ lies in $\cO_{K}$ since $1+tb$ is a unit of $\cO_{K}$.
\end{proof}

We now prove the two key lemmas for the characterization: the first one extracts from the Kochen operator the discreteness of the value group, and the second one converts discreteness into a $\Presf$-structure.

\begin{lemma}\label{lem:gap}
Let $w$ be a valuation on a field $L$ with $w(t)>0$ and $w(2)=0$, and suppose $\gam(x)\in\cO_{w}$ for every $x\in L$ with $1+tx^{2}\neq0$. Then $w(t)$ is the smallest positive element of $\Gamma_{w}$.
\end{lemma}

\begin{proof}
It suffices to show that no element of $\Gamma_{w}$ lies in the open interval $(-w(t),0)$: if $0<\gamma<w(t)$ for some $\gamma\in\Gamma_{w}$, then $-\gamma\in(-w(t),0)$.

Suppose then that $\ell\in L$ satisfies $-w(t)<w(\ell)<0$.

Assume first $1+t\ell^{2}=0$. Then $\ell\ne-\tfrac12$, since $w(-\tfrac12)=-w(2)=0$. Set $\ell'=\ell+1$; as $w(\ell)<0=w(1)$ we get $w(\ell')=w(\ell)\in(-w(t),0)$, and
\[
1+t(\ell')^{2}=\underbrace{1+t\ell^{2}}_{=0}+\,t(2\ell+1)=t(2\ell+1)\neq0 .
\]
Moreover $w(2\ell+1)=w(\ell)$ because $w(2\ell)=w(\ell)<0=w(1)$. Hence
\[
w(\gam(\ell'))=w(\ell')-w(t)-w(2\ell+1)=w(\ell)-w(t)-w(\ell)=-w(t)<0,
\]
contradicting $\gam(\ell')\in\cO_{w}$.

Assume now $1+t\ell^{2}\neq0$ and write $\gam(\ell)=\ell/(1+t\ell^{2})$. If $w(t)+2w(\ell)<0$ then $w(1+t\ell^{2})=w(t)+2w(\ell)$ and
\[
w(\gam(\ell))=w(\ell)-w(t)-2w(\ell)=-w(\ell)-w(t)<0,
\]
since $w(\ell)>-w(t)$. If instead $w(t)+2w(\ell)\geqslant0$ then both $1$ and $t\ell^{2}$ have non-negative value, so $w(1+t\ell^{2})\geqslant0$ and
\[
w(\gam(\ell))=w(\ell)-w(1+t\ell^{2})\leqslant w(\ell)<0 .
\]
In both cases $\gam(\ell)\notin\cO_{w}$, a contradiction.
\end{proof}

\begin{lemma}\label{lem:expansion}
Let $\Gamma$ be an ordered abelian group and $\varepsilon\in\Gamma$ with $\varepsilon>0$. Then $(\Gamma,\varepsilon)$ can be expanded by predicates $(D_{n})_{n\geqslant2}$ to a model of $\Presf$ (with $1$ interpreted as $\varepsilon$) if and only if $\varepsilon$ is the smallest positive element of $\Gamma$.
\end{lemma}

\begin{proof}
($\Rightarrow$) The sentence $\forall x\,(x\leqslant0\vee x\geqslant1)$ is universal and true in $\mathbf{Z}$, hence belongs to $\Presf$.

($\Leftarrow$) Assume $\varepsilon$ is the smallest positive element of $\Gamma$. It suffices to embed $(\Gamma,\varepsilon)$, as an ordered abelian group with a distinguished element, into a $\Z$-group $M$ with $\varepsilon$ mapping to the smallest positive element of $M$: setting $D_{n}=nM\cap\Gamma$ then turns $(\Gamma,\varepsilon,(D_{n})_{n})$ into an $\cL_{\Pres}$-substructure of $M\models\Pres$, hence into a model of the universal theory $\Presf$.

By compactness, it is enough to prove that $\Pres\cup\Delta(\Gamma,\varepsilon)$ is consistent, where $\Delta(\Gamma,\varepsilon)$ is the atomic diagram of $(\Gamma,\varepsilon)$ in the language $\{+,-,<,0\}$ with constants for the elements of $\Gamma$, the constant for $\varepsilon$ being identified with $1$. A finite subset of this theory mentions finitely many elements $\gamma_{1},\dots,\gamma_{k}\in\Gamma$, and quantifier-free formulas about them are absolute between $\Gamma$ and the subgroup $G=\langle\gamma_{1},\dots,\gamma_{k},\varepsilon\rangle$ with the induced order. So it suffices to embed each such $G$ into a $\Z$-group, with $\varepsilon$ going to the smallest positive element.

\emph{Step 1: $\Z\varepsilon$ is convex in $G$.} We check, by induction on $m\geqslant1$, that every $\gamma\in G$ with $0<\gamma\leqslant m\varepsilon$ lies in $\Z\varepsilon$. For $m=1$, minimality of $\varepsilon$ in $\Gamma$ gives $\gamma\geqslant\varepsilon$, so $\gamma=\varepsilon$. If $0<\gamma\leqslant(m+1)\varepsilon$, then again $\gamma\geqslant\varepsilon$, so $0\leqslant\gamma-\varepsilon\leqslant m\varepsilon$, and either $\gamma=\varepsilon$ or, by induction, $\gamma-\varepsilon\in\Z\varepsilon$.

\emph{Step 2:} Since $\Z\varepsilon$ is convex, the quotient $\bar{G}=G/\Z\varepsilon$ is an ordered abelian group, with $\bar{g}>0$ if and only if $g>0$ and $g\notin\Z\varepsilon$. Being finitely generated and torsion-free, $\bar{G}$ is free abelian (this uses only the abstract group structure), so the exact sequence $0\to\Z\varepsilon\to G\to\bar{G}\to0$ splits: there is a subgroup $G_{0}\leqslant G$ with $G=G_{0}\oplus\Z\varepsilon$. For $g=g_{0}+m\varepsilon$ with $g_{0}\in G_{0}$ and $m\in\Z$, the sign of $g$ is the sign of $\bar{g}=\bar{g_{0}}$ when $\bar{g_{0}}\neq0$, and the sign of $m$ otherwise, in which case $g_{0}=0$; hence $g\mapsto(\bar{g_{0}},m)$ is an isomorphism of ordered groups between $G$ and $\bar{G}\times\Z$ ordered lexicographically, sending $\varepsilon$ to $(0,1)$.

\emph{Step 3:} Let $D=\bar{G}\otimes_{\Z}\Q$ be the divisible hull of $\bar{G}$, with its unique compatible order, and let $M=D\times\Z$, ordered lexicographically. Then $M$ is discretely ordered with smallest positive element $(0,1)$, and for every $n\geqslant2$ and $(d,m)\in M$ we have $(d,m)-(0,r)=n\cdot(d/n,(m-r)/n)\in nM$, where $r$ is the remainder of $m$ modulo $n$; moreover $(0,r)\in nM$ forces $n\mid r$. Hence $M/nM$ is cyclic of order $n$ generated by $(0,1)+nM$, and $M$ is a $\Z$-group. By Step 2, $G$ embeds into $M$ as an ordered group with $\varepsilon\mapsto(0,1)$, which completes the proof.
\end{proof}

We can now prove Theorem~A.

\begin{theorem}\label{thm:characterization}
Let $L$ be a field extension of $K$. The following are equivalent:
\begin{enumerate}[label=\textnormal{(\arabic*)}]
\item $L$ is formally $T$-adic over $K$;
\item $L$ admits a $\tf$-valuation;
\item $1/t\notin\cO_{K}[\gam(L)]$;
\item $1/t\notin\Lam_{L}$.
\end{enumerate}
\end{theorem}

\begin{proof}
(1)$\Leftrightarrow$(2) is \cref{prop:link}, and (4)$\Rightarrow$(3) is clear since $\cO_{K}[\gam(L)]\subseteq\Lam_{L}$.

(2)$\Rightarrow$(4). Let $w$ be a $\tf$-valuation. Then $w(t)$ is the smallest positive element of $\Gamma_{w}$, so by \cref{prop:gamma-into-O}(3), $\Lam_{L}\subseteq\cO_{w}$. Since $w(1/t)=-w(t)<0$, we get $1/t\notin\Lam_{L}$.

(3)$\Rightarrow$(2). Let $R=\cO_{K}[\gam(L)]$. The ideal $tR$ is proper: otherwise $1=t\rho$ for some $\rho\in R$, giving $1/t\in R$. Let $P$ be a maximal (hence prime) ideal of $R$ containing $tR$. By Chevalley's extension theorem \cite[Theorem 3.1.1]{EP05} there is a valuation ring $\cO_{w}$ of $L$ with $R\subseteq\cO_{w}$ and $\cM_{w}\cap R=P$; in particular $t\in\cM_{w}$, i.e.\ $w(t)>0$.

We claim that $w$ extends $v$. First, $\cO_{K}\subseteq R\subseteq\cO_{w}$. The set $q=\cM_{w}\cap\cO_{K}$ is a prime ideal of $\cO_{K}$ containing $t$. Since $v(t)$ is the smallest positive element of $\Gamma_{K}$, every nonzero convex subgroup of $\Gamma_{K}$ contains $v(t)$; as the prime ideals of $\cO_{K}$ are exactly the sets $\{x: v(x)>\Delta\}$ for $\Delta$ a convex subgroup of $\Gamma_{K}$, the only prime of $\cO_{K}$ containing $t$ is $\cM_{K}$. Hence $q=\cM_{K}$. Now $\cO'=\cO_{w}\cap K$ is a valuation ring of $K$ containing $\cO_{K}$, so $\cO'=(\cO_{K})_{q'}$ for the prime $q'=(\cM_{w}\cap K)\cap\cO_{K}=q=\cM_{K}$ \cite[Lemma 2.3.1]{EP05}; thus $\cO'=\cO_{K}$ and $w$ extends $v$.

In particular $w(2)=v(2)=0$, and $\gam(x)\in R\subseteq\cO_{w}$ for every $x\in L$ at which $\gam$ is defined. By the gap lemma (\cref{lem:gap}), $w(t)$ is the smallest positive element of $\Gamma_{w}$, and by the expansion lemma (\cref{lem:expansion}) the pair $(\Gamma_{w},w(t))$ expands to a model of $\Presf$. Hence $w$ is a $\tf$-valuation.
\end{proof}

\begin{examples}\label{ex:sanity}
\leavevmode
\begin{enumerate}[label=\textnormal{(\arabic*)}, leftmargin=*]
\item $L=\KrX$: by \cref{cor:gauss-T-adic}, $L$ is formally $T$-adic, so $1/t\notin\Lam_{\KrX}$. This fact will be used crucially in the proof of the Nullstellensatz.
\item $L=\C(\!(t^{1/2})\!)$ over $K=\ct$: the natural $t^{1/2}$-adic valuation $w$ has $w(t)=2w(t^{1/2})$, so it is not a $\tf$-valuation, and one may suspect that none exists. \Cref{thm:characterization} confirms this through a computation:
\[
\gam(t^{-1/2})=\frac{t^{-1/2}}{1+t\cdot t^{-1}}=\frac{t^{-1/2}}{2},
\]
so $t^{-1/2}=2\gam(t^{-1/2})\in\cO_{K}[\gam(L)]$ and $1/t=(t^{-1/2})^{2}\in\cO_{K}[\gam(L)]$; hence $L$ is not formally $T$-adic.
\item $L=k(I)$ for $I=(X^{2}+t)\subseteq K[X]$ ($n=1$): the polynomial $X^{2}+t$ is irreducible ($-t$ is not a square in $K$ since $v(t)\notin2\Gamma_{K}$), and $k(I)=K(\theta)$ with $\theta^{2}=-t$. Any $\tf$-valuation $w$ would give $w(t)=2w(\theta)\in2\Gamma_{w}$, which is impossible; so $k(I)$ is not formally $T$-adic and, by \cref{thm:characterization}, $1/t\in\Lam_{k(I)}$. Explicitly, a computation shows
\[
\gam(1+\theta^{-1})+\gam(1-\theta^{-1})=\frac{1+\theta}{t(2+\theta)}+\frac{\theta-1}{t(\theta-2)}=\frac{2(t+2)}{t(t+4)},
\]
whence $1/t=\frac{t+4}{2(t+2)}\big(\gam(1+\theta^{-1})+\gam(1-\theta^{-1})\big)\in\cO_{K}[\gam(k(I))]$, the factor $\frac{t+4}{2(t+2)}$ being a unit of $\cO_{K}$.
\end{enumerate}
\end{examples}

\section{The Nullstellensatz}\label{sec:nss}

We keep the notation of the previous sections: $\mathbf{K}\models T$, and $I$ denotes an ideal of $\KX$, not necessarily proper, with fixed generators $f_{1},\dots,f_{l}$ (recall that $\KX$ is noetherian).

\begin{definition}\label{def:T-radical}
Set $R=\Lam_{\KrX}\cdot\KX$, the subring of $\KrX$ consisting of all finite sums $\sum_{i}\lambda_{i}g_{i}$ with $\lambda_{i}\in\Lam_{\KrX}$ and $g_{i}\in\KX$; note that $\KX\subseteq R\subseteq\KrX$. The radical relevant to our Nullstellensatz is the $R$-radical $\rt[R]{I}$ of \cref{def:A-radical}. We say that $I$ is an \emph{$R$-ideal} if $\rt[R]{I}=I$.
\end{definition}


The first step is to show that elements of $R$ can be evaluated at every point of $K^{n}$, via the evaluation places of \cref{def:eval-places}.

\begin{lemma}\label{lem:places-R}
Let $\varphi\colon\KrX_{\infty}\to K_{\infty}$ be a place with $\varphi|_{K}=\mathrm{id}_{K}$. Then:
\begin{enumerate}[label=\textnormal{(\arabic*)}]
\item $\varphi(\gam(h))\in\cO_{K}$ for every $h\in\KrX$ with $1+th^{2}\neq0$;
\item $\varphi\big(\cO_{K}[\gam(\KrX)]\big)\subseteq\cO_{K}$ and $\varphi(\Lam_{\KrX})\subseteq K$;
\item if moreover $\varphi(X_{i})\neq\infty$ for $1\leqslant i\leqslant n$ --- as is the case for every evaluation place $P_{a,\sigma}$ --- then $\varphi(\omega)\neq\infty$ for every $\omega\in R$.
\end{enumerate}
\end{lemma}

\begin{proof}
Throughout we use freely the axioms of places: sums and products of elements whose $\varphi$-values lie in $K$ are evaluated coordinatewise (the corresponding sums and products of values being defined), and $\varphi(h^{-1})=\varphi(h)^{-1}$ whenever $\varphi(h)\in K^{\times}\cup\{0,\infty\}$ (see \cref{subsec:valued}).

(1) Suppose first $\varphi(h)\neq\infty$. Then $\varphi(1+th^{2})=1+t\varphi(h)^{2}$, which is a nonzero element of $K$ because $1+ta^{2}\neq0$ for every $a\in K$ (\cref{prop:gamma-into-O}(1) applied to $(K,v)$). Hence
\[
\varphi(\gam(h))=\varphi(h)\cdot\big(1+t\varphi(h)^{2}\big)^{-1}=\gam(\varphi(h))\in\gam(K)=\cO_{K},
\]
using \cref{prop:gamma-K} for the last equality.

Suppose now $\varphi(h)=\infty$; in particular $h\neq0$ and, in $\KrX$,
\[
\gam(h)=\frac{h}{1+th^{2}}=\big(h^{-1}+th\big)^{-1}.
\]
Then $\varphi(h^{-1})=0$, while $\varphi(th)=t\cdot\varphi(h)=\infty$ since $t\neq0$. Hence $\varphi(h^{-1}+th)=0+\infty=\infty$, the sum of values being defined, and therefore $\varphi(\gam(h))=\infty^{-1}=0\in\cO_{K}$.

(2) An element of $\cO_{K}[\gam(\KrX)]$ is a polynomial expression, with coefficients in $\cO_{K}$ fixed by $\varphi$, in finitely many elements $\gam(h_{j})$, whose $\varphi$-values lie in $\cO_{K}$ by (1); coordinatewise evaluation keeps the value of the whole expression in the subring $\cO_{K}$. For $\lambda=a/(1+tb)\in\Lam_{\KrX}$ we have $\varphi(1+tb)=1+t\varphi(b)$ with $\varphi(b)\in\cO_{K}$, so $v\big(t\varphi(b)\big)>0$ and $1+t\varphi(b)\in\cO_{K}^{\times}$; hence $\varphi(\lambda)=\varphi(a)\cdot\big(1+t\varphi(b)\big)^{-1}\in K$.

(3) If $\varphi(X_{i})\in K$ for all $i$, then $\varphi(g)=g\big(\varphi(X_{1}),\dots,\varphi(X_{n})\big)\in K$ for every $g\in\KX$, again by coordinatewise evaluation. An element of $R$ is a finite sum $\sum_{i}\lambda_{i}g_{i}$ with $\varphi(\lambda_{i})\in K$ by (2) and $\varphi(g_{i})\in K$, so its $\varphi$-value lies in $K$. Finally, an evaluation place satisfies $P_{a,\sigma}(X_{i})=a_{i}\in K$.
\end{proof}

\begin{corollary}\label{cor:JV-T-ideal}
$J(Z(I))$ is an $R$-ideal.
\end{corollary}

\begin{proof}
If $J(Z(I))=\KX$ the claim is trivial, since $\rt[R]{\KX}=\KX$. Otherwise let $h_{1},\dots,h_{N}$ generate $J(Z(I))$ and let $g\in\rt[R]{J(Z(I))}$, say $g^{m}=\sum_{j=1}^{N}\omega_{j}h_{j}$ with $\omega_{j}\in R$. Fix $a\in Z(I)$ and any permutation $\sigma$, and apply the place $\varphi=P_{a,\sigma}$ to this identity in $\KrX$: since $\varphi(\omega_{j})\in K$ by \cref{lem:places-R}(3), $\varphi(h_{j})=h_{j}(a)=0$ and $\varphi(g^{m})=g(a)^{m}$, coordinatewise evaluation gives
\[
g(a)^{m}=\sum_{j=1}^{N}\varphi(\omega_{j})\cdot0=0 .
\]
Hence $g(a)=0$ for every $a\in Z(I)$, i.e.\ $g\in J(Z(I))$.
\end{proof}

\begin{proposition}\label{prop:T-radical-props}
Let $I,I'$ be ideals of $\KX$.
\begin{enumerate}[label=\textnormal{(\arabic*)}]
\item If $I\subseteq I'$ and $I'$ is an $R$-ideal, then $\rt[R]{I}\subseteq I'$; in other words, $\rt[R]{I}$ is contained in every $R$-ideal containing $I$.
\item $\rt[R]{I}$ is an $R$-ideal; consequently $\rt[R]{I}$ is the smallest $R$-ideal containing $I$.
\item Every prime ideal of $\KX$ minimal over an $R$-ideal is an $R$-ideal.
\end{enumerate}
\end{proposition}

\begin{proof}
(1) If $g\in\rt[R]{I}$ then $g^{m}\in IR\subseteq I'R$ for some $m$, so $g\in\KX\cap\sqrt{I'R}=\rt[R]{I'}=I'$.

(2) Since $\rt[R]{I}\subseteq\sqrt{IR}$, we get $\big(\rt[R]{I}\big)R\subseteq\sqrt{IR}$, hence
\[
\sqrt{\big(\rt[R]{I}\big)R}\subseteq\sqrt{\sqrt{IR}}=\sqrt{IR},
\]
and intersecting with $\KX$ yields $\rt[R]{\rt[R]{I}}\subseteq\rt[R]{I}$; the other inclusion always holds.

(3) Let $I$ be an $R$-ideal, without loss of generality proper, and write $I=P_{1}\cap\dots\cap P_{k}$ where the $P_{j}$ are the primes minimal over the radical ideal $I$, with $P_{j}\not\subseteq P_{j'}$ for $j\neq j'$. We show that $P_{1}$ is an $R$-ideal; the same argument applies to each $P_{j}$. If $k=1$ there is nothing to prove, so assume $k\geqslant2$, choose $h_{j}\in P_{j}\setminus P_{1}$ for $j\geqslant2$ and set $h=h_{2}\cdots h_{k}$; then $h\in\bigcap_{j\geqslant2}P_{j}$ and $h\notin P_{1}$ since $P_{1}$ is prime.

Let $q_{1},\dots,q_{N}$ generate $P_{1}$ and let $g\in\rt[R]{P_{1}}$, say $g^{m}=\sum_{j}\mu_{j}q_{j}$ with $\mu_{j}\in R$. For each $j$ we have $hq_{j}\in P_{1}\cap\bigcap_{j'\geqslant2}P_{j'}=I$, so $hq_{j}=\sum_{i}\alpha_{ji}f_{i}$ with $\alpha_{ji}\in\KX$, where $f_{1},\dots,f_{l}$ generate $I$. Then
\[
(hg)^{m}=h^{m}g^{m}=\sum_{j}\mu_{j}h^{m-1}\,(hq_{j})=\sum_{i}\Big(\sum_{j}\mu_{j}h^{m-1}\alpha_{ji}\Big)f_{i}\in{IR},
\]
since $\mu_{j}h^{m-1}\alpha_{ji}\in R$. Hence $hg\in\rt[R]{I}=I\subseteq P_{1}$, and since $h\notin P_{1}$ and $P_{1}$ is prime, $g\in P_{1}$.
\end{proof}

The key step is the following analogue of \cite[Proposition 3.6]{srhir}: for prime ideals, being an $R$-ideal is equivalent to $k(I)$ being formally $T$-adic. One implication is where \cref{thm:characterization} enters.

\begin{theorem}\label{thm:prime-T-ideal}
Let $I\subsetneq\KX$ be a prime ideal. If $I$ is an $R$-ideal, then $k(I)$ is formally $T$-adic over $K$.
\end{theorem}

\begin{proof}
Write $\pi\colon\KX_{I}\to k(I)$ for the canonical map, where $\KX_{I}\subseteq\KrX$ denotes the localization of $\KX$ at $I$, and note that $\pi$ restricts to the identity on $K$ and $\pi(t)=t$.

Suppose towards a contradiction that $k(I)$ is not formally $T$-adic. By \cref{thm:characterization} (applied to the field extension $k(I)$ of $K$) we have $1/t\in\Lam_{k(I)}$, say
\begin{equation}\label{eq:one-over-t}
\frac{1}{t}=\frac{\bar a}{1+t\bar b},\qquad \bar a,\bar b\in\cO_{K}[\gam(k(I))],\ 1+t\bar b\neq0 .
\end{equation}

\emph{Step 1:} The element $\bar a$ is a polynomial expression, with coefficients in $\cO_{K}$, in finitely many values $\gam(\xi_{j})$ with $\xi_{j}\in k(I)$ and $1+t\xi_{j}^{2}\neq0$. Write $\xi_{j}=\pi(p_{j}/q_{j})$ with $p_{j},q_{j}\in\KX$ and $q_{j}\notin I$, and set $d_{j}=q_{j}^{2}+tp_{j}^{2}\in\KX$. Its image satisfies
\[
\pi(d_{j})=\bar q_{j}^{2}\,(1+t\xi_{j}^{2})\neq0,
\]
so $d_{j}\notin I$; in particular $d_{j}\neq0$ and the element
\[
\eta_{j}:=\gam\big(p_{j}/q_{j}\big)=\frac{p_{j}q_{j}}{d_{j}}\in\KX_{I}
\]
is defined, belongs to $\gam(\KrX)$, and satisfies $\pi(\eta_{j})=\gam(\xi_{j})$. Replacing each $\gam(\xi_{j})$ by $\eta_{j}$ in the polynomial expression for $\bar a$, we obtain $a\in\cO_{K}[\gam(\KrX)]\cap\KX_{I}$ with $\pi(a)=\bar a$; moreover, taking as common denominator a suitable product $G_{1}$ of powers of the $d_{j}$, we may write $a=F_{1}/G_{1}$ with $F_{1},G_{1}\in\KX$ and $G_{1}\notin I$ (a product of elements outside the prime ideal $I$). Similarly we lift $\bar b$ to $b=F_{2}/G_{2}\in\cO_{K}[\gam(\KrX)]\cap\KX_{I}$ with $G_{2}\notin I$ and $\pi(b)=\bar b$.

\emph{Step 2:}
Equation \eqref{eq:one-over-t} reads $1+t\bar b=t\bar a$ in $k(I)$. Substituting $\bar a=\pi(F_{1})/\pi(G_{1})$ and $\bar b=\pi(F_{2})/\pi(G_{2})$ and multiplying through by $\pi(G_{1}G_{2})$, we obtain
\[
\pi\big(G_{1}G_{2}+t(F_{2}G_{1}-F_{1}G_{2})\big)=0 .
\]
Hence $c:=G_{1}G_{2}+t(F_{2}G_{1}-F_{1}G_{2})\in I$, and we may write $c=\sum_{i=1}^{l}\alpha_{i}f_{i}$ with $\alpha_{i}\in\KX$.

\emph{Step 3:}
Consider the element of $\KrX$
\[
u:=1+t(b-a) .
\]
Note that $b-a\in\cO_{K}[\gam(\KrX)]$ and that
\[
u=1+t\Big(\frac{F_{2}}{G_{2}}-\frac{F_{1}}{G_{1}}\Big)=\frac{G_{1}G_{2}+t(F_{2}G_{1}-F_{1}G_{2})}{G_{1}G_{2}}=\frac{c}{G_{1}G_{2}} .
\]
We claim $u\neq0$. Otherwise $1/t=a-b\in\cO_{K}[\gam(\KrX)]$, so $\KrX$ would not be formally $T$-adic by \cref{thm:characterization}, contradicting \cref{cor:gauss-T-adic}. Consequently $1/u=1/(1+t(b-a))\in\Lam_{\KrX}$.

\emph{Step 4:}
Dividing $c=\sum_{i}\alpha_{i}f_{i}$ by $u$ yields
\[
G_{1}G_{2}=\frac{c}{u}=\sum_{i=1}^{l}\frac{\alpha_{i}}{u}\,f_{i},\qquad \frac{\alpha_{i}}{u}\in\Lam_{\KrX}\cdot\KX=R .
\]
Hence $G_{1}G_{2}\in\rt[R]{I}=I$. But $G_{1}\notin I$ and $G_{2}\notin I$, contradicting the primality of $I$.
\end{proof}

\begin{corollary}\label{cor:T-ideal-closed}
Every $R$-ideal of $\KX$ is closed. Moreover, for a prime ideal $I\subsetneq\KX$ the following are equivalent: \textnormal{(i)} $I$ is an $R$-ideal; \textnormal{(ii)} $k(I)$ is formally $T$-adic over $K$; \textnormal{(iii)} $I$ is closed.
\end{corollary}

\begin{proof}
We first prove the equivalence for prime $I$. The implication (i)$\Rightarrow$(ii) is \cref{thm:prime-T-ideal} and (ii)$\Rightarrow$(iii) is \cref{thm:model-implies-closed} applied with $\cL_{0}=\cL$ and $T_{0}=T$ (model complete by \cref{thm:qe}). Finally, for (iii)$\Rightarrow$(i), if $I$ is closed then $I=J(Z(I))$ is an $R$-ideal by \cref{cor:JV-T-ideal}.

Now let $I$ be any $R$-ideal; if $I=\KX$ then $Z(I)=\emptyset$ and $J(Z(I))=\KX$, so $I$ is closed. Otherwise, since $I$ is radical, $I=P_{1}\cap\dots\cap P_{k}$ with $P_{j}$ the minimal primes over $I$; each $P_{j}$ is an $R$-ideal by \cref{prop:T-radical-props}(3), hence closed by the equivalence above. Therefore
\[
J(Z(I))=J\Big(\bigcup_{j}Z(P_{j})\Big)=\bigcap_{j}J(Z(P_{j}))=\bigcap_{j}P_{j}=I .\qedhere
\]
\end{proof}

\begin{theorem}[Nullstellensatz for $\ct$]\label{thm:nss}
Let $\mathbf{K}\models T$ and let $I=\langle f_{1},\dots,f_{l}\rangle$ be an ideal of $\KX$. Then
\[
J(Z(I))=\rt[R]{I}.
\]
Equivalently, a polynomial $f\in\KX$ vanishes on $Z(I)$ if and only if there are $m\geqslant1$ and $\omega_{1},\dots,\omega_{l}\in R=\Lam_{\KrX}\cdot\KX$ such that
\[
f^{m}=\sum_{i=1}^{l}\omega_{i}f_{i}.
\]
\end{theorem}

\begin{proof}
By \cref{cor:JV-T-ideal}, $J(Z(I))$ is an $R$-ideal containing $I$, so $\rt[R]{I}\subseteq J(Z(I))$ by \cref{prop:T-radical-props}(1). Applying the inclusion-reversing operator $Z$ to $I\subseteq\rt[R]{I}\subseteq J(Z(I))$ and using $Z(J(Z(I)))=Z(I)$ (\cref{lem:JV-props}(1) and its proof) gives $Z(\rt[R]{I})=Z(I)$. Since $\rt[R]{I}$ is an $R$-ideal (\cref{prop:T-radical-props}(2)), it is closed by \cref{cor:T-ideal-closed}, so
\[
\rt[R]{I}=J\big(Z(\rt[R]{I})\big)=J(Z(I)).\qedhere
\]
\end{proof}

\begin{corollary}\label{cor:bijections}
The operators $Z$ and $J$ induce mutually inverse bijections between:
\begin{enumerate}[label=\textnormal{(\arabic*)}]
\item $R$-ideals of $\KX$ and algebraic subsets of $K^{n}$;
\item prime $R$-ideals and irreducible algebraic subsets of $K^{n}$;
\item maximal proper $R$-ideals and points of $K^{n}$.
\end{enumerate}
\end{corollary}

\begin{proof}
By \cref{cor:JV-T-ideal,cor:T-ideal-closed}, the $R$-ideals of $\KX$ are exactly the closed ideals. The statement is then \cref{thm:bijections-abstract} applied with $\cL_{0}=\cL$ and $T_{0}=T$.
\end{proof}

\begin{example}\label{ex:nss-example}
Let $I=(X^{2}+t)\subseteq K[X]$, a prime ideal with $Z(I)=\emptyset$ (\cref{ex:sanity}(3)). \Cref{thm:nss} asserts that $\rt[R]{I}=J(\emptyset)=K[X]$, i.e.\ that $1/(X^{2}+t)\in R$. An explicit certificate is easy to exhibit here: since
\[
\gam(1/X)=\frac{1/X}{1+t/X^{2}}=\frac{X}{X^{2}+t}
\qquad\text{and hence}\qquad
1-X\,\gam(1/X)=\frac{t}{X^{2}+t},
\]
we get
\[
1=\underbrace{\tfrac{1}{t}\big(1-X\gam(1/X)\big)}_{\textstyle\in\,R}\cdot\,(X^{2}+t),
\]
noting that $t^{-1}\in K\subseteq\KX\subseteq R$ and $X\gam(1/X)\in\KX\cdot\Lam_{\KrX}\subseteq R$. This illustrates a feature of the ring $R=\Lam_{\KrX}\cdot\KX$: although $1/t\notin\Lam_{\KrX}$ (\cref{cor:gauss-T-adic}), the polynomial factor allows arbitrary coefficients from $K$; what $R$ cannot do is reach the poles of $\gam$-denominators at $K$-rational points, by \cref{lem:places-R}.
\end{example}

\section{The \texorpdfstring{$T$}{T}-spectrum}\label{sec:spectrum}

In this final section we compare three natural incarnations of a ``spectrum'' attached to $T$, in the spirit of the real spectrum and of the $p$-adic spectrum \cite{Robinson86,Belair90,Srhir08}: a set of prime ideals, a set of morphisms into models of $T$ and, for polynomial rings, the space of types of $T$. Throughout, $\mathbf{K}\models T$ is fixed and $A$ denotes a commutative $K$-algebra.

If $M$ is the underlying field of a model $\mathbf{M}\models T$ equipped with a $K$-algebra structure, the structure map $K\to M$ is a ring embedding sending $t$ to $t$, hence an elementary $\cL$-embedding by \cref{thm:qe}(4); we will use this repeatedly. For a $K$-algebra homomorphism $\alpha\colon A\to M$, we write $\langle\alpha(A)\rangle_{\mathbf{M}}$ for the $\cL$-substructure of $\mathbf{M}$ generated by $\alpha(A)$, i.e.\ the subfield of $M$ generated by $\alpha(A)$ over $K$, equipped with the restricted predicates $P_{n}$.

\begin{definition}\label{def:T-spec} \
\begin{enumerate}[label=\textnormal{(\arabic*)}, leftmargin=*]
\item The \emph{$T$-spectrum} of $A$ is
\[
\tSpec(A)=\big\{\alpha\colon A\to M\ \text{$K$-algebra homomorphism},\ \mathbf{M}\models T\big\}\big/\sim\;,
\]
where $\alpha\sim\beta$ (for $\alpha\colon A\to M$, $\beta\colon A\to N$) if and only if there is an isomorphism of $\cL$-structures $\gam_{0}\colon\langle\alpha(A)\rangle_{\mathbf{M}}\to\langle\beta(A)\rangle_{\mathbf{N}}$ with $\gam_{0}\circ\alpha=\beta$. We write $[\alpha]$ for the $\sim$-class of $\alpha$. For tuples $a=(a_{0},a_{1},\dots,a_{m})\in A^{m+1}$ and $n=(n_{1},\dots,n_{m})$ with $n_{i}\geqslant2$, set
\[
U(a,n)=\Big\{[\alpha]\in\tSpec(A)\ :\ \alpha(a_{0})\neq0\ \wedge\ \bigwedge_{i=1}^{m}\mathbf{M}_{\alpha}\models P_{n_{i}}\big(\alpha(a_{i})\big)\Big\},
\]
where $\mathbf{M}_{\alpha}$ denotes the target model of $\alpha$. The \emph{spectral topology} on $\tSpec(A)$ is the topology having these sets as a basis.
\item We let 
\[
\SpecT(A)=\big\{\mathfrak{p}\in\Spec(A)\,:\, \Frac(A/\mathfrak{p})\ \text{is formally $T$-adic over }K\big\},
\]
where $\Frac(A/\mathfrak p)$ is viewed as a field extension of $K$ via $K\to A\to A/\mathfrak{p}$. We equip $\SpecT(A)$ with the subspace topology from the Zariski topology on $\Spec(A)$.
\item We let $S_{n}(K)$ denote the Stone space of complete $n$-types over $K$ equipped with the logic topology. 
\end{enumerate}
\end{definition}

\begin{remark}\label{rem:amalgamation}
Note that $\sim$ is indeed an equivalence relation. Moreover, by quantifier elimination, every isomorphism between $\cL$-substructures of models of $T$ is a partial elementary map, and therefore extends to an elementary embedding into any sufficiently saturated elementary extension of the target model. It follows that $\alpha\sim\beta$ if and only if there are $\mathbf{E}\models T$ and $\cL$-embeddings $g\colon\mathbf{M}\to\mathbf{E}$, $h\colon\mathbf{N}\to\mathbf{E}$ with $g\circ\alpha=h\circ\beta$. In particular, our definition of $\tSpec(A)$ agrees with the one used by Srhir in the $p$-adic case \cite[\S2]{Srhir08}.
\end{remark}

\begin{theorem}\label{thm:types}\
\begin{enumerate}[label=\textnormal{(\arabic*)}, leftmargin=*]
\item The assignment $[\alpha]\mapsto\ker\alpha$ is a well-defined continuous surjection
\[
\ker\colon \tSpec(A)\longrightarrow\SpecT(A),
\]
which in general is not injective (see \cref{ex:kernel-not-injective}).
\item For $A=\KX$, the map $\Phi\colon S_{n}(K)\to\tSpec(\KX)$ sending $p$ to the class of
\[
\alpha_{p}\colon\KX\to M,\qquad f\mapsto f(c),
\]
for any realization $c$ of $p$ in an elementary extension $\mathbf{M}\succeq\mathbf{K}$, is a continuous bijection. Moreover, $\ker\circ\,\Phi$ sends $p=\tp(c/K)$ to $I_{c}=\{f\in\KX:f(c)=0\}$, and $\SpecT(\KX)$ is exactly the set of prime $R$-ideals of $\KX$.
\end{enumerate}
\end{theorem}

\begin{proof}
(1) If $\alpha\sim\beta$ via $\gam_{0}$ then $\Ker\beta=\Ker(\gam_{0}\circ\alpha)=\Ker\alpha$, $\gam_{0}$ being injective; so the map is well defined. For $\alpha\colon A\to M$, the kernel $\mathfrak p=\Ker\alpha$ is prime and $\Frac(A/\mathfrak p)$ is $K$-isomorphic to the underlying field of $\langle\alpha(A)\rangle_{\mathbf{M}}$, an $\cL$-substructure of $\mathbf{M}\models T$ containing $\mathbf{K}$ as an $\cL$-substructure; hence $\Frac(A/\mathfrak p)$ is formally $T$-adic over $K$ and $\mathfrak p\in\SpecT(A)$. Conversely, given $\mathfrak p\in\SpecT(A)$, equip $\Frac(A/\mathfrak p)$ with a structure modeling $\tf$ over $\mathbf{K}$ and embed it into some $\mathbf{M}\models T$; the composite $\alpha_{\mathfrak p}\colon A\to A/\mathfrak p\to\Frac(A/\mathfrak p)\to M$ is a $K$-algebra homomorphism with kernel $\mathfrak p$, so $\ker$ is surjective. Finally, taking $m=0$, the basic opens $U(a_{0})$ are the preimages of the Zariski opens $D(a_{0})$, so $\ker$ is continuous for the Zariski topology on $\SpecT(A)$. 

(2) Let us first show that $\Phi$ is well defined. Let $c,c'$ be realizations of $p$ in some elementary extension $\mathbf{M}$. The substructure $\langle\alpha_{p}(\KX)\rangle_{\mathbf{M}}$ is the subfield $K(c)$ with the induced predicates, and likewise for $c'$. Since $\tp(c/K)=\tp(c'/K)$, the assignment $f(c)/g(c)\mapsto f(c')/g(c')$ is a well-defined field isomorphism $K(c)\to K(c')$ over $K$ commuting with $\alpha_{p},\alpha_{p}'$, and it preserves each $P_{n}$ because the conditions $P_{n}\big(f(x)g(x)^{-1}\big)$ are formulas over $K$ belonging to both types. For injectivity, if $[\alpha_{p}]=[\alpha_{q}]$, with realizations $c\models p$ and $c'\models q$, there is an $\cL$-isomorphism $\gam_{0}\colon K(c)\to K(c')$ over $K$ with $\gam_{0}(c)=c'$; hence $p=q$. Finally, for surjectivity, let $\alpha\colon\KX\to M$ be a $K$-algebra homomorphism with $\mathbf{M}\models T$. As noted above we may assume $\mathbf{K}\preceq\mathbf{M}$. Set $c=(\alpha(X_{1}),\dots,\alpha(X_{n}))$ and $p=\tp(c/K)$; since $\alpha$ is a $K$-algebra map, $\alpha=\alpha_{p}$ and $[\alpha]=\Phi(p)$. Continuity follows immediately since the basic opens $U(a,m)$ are definable sets. 

Finally, $\Ker\alpha_{p}=I_{c}$ by definition, and $\SpecT(\KX)$ consists of the primes $\mathfrak p$ with $k(\mathfrak p)$ formally $T$-adic, which by \cref{cor:T-ideal-closed} are exactly the prime $R$-ideals.
\end{proof}

\begin{example}\label{ex:kernel-not-injective}
Let $A=K[x]$ be a polynomial ring in one variable and let $\mathbf{M}\succeq\mathbf{K}$ be an elementary extension containing a unit $u\in\cO_{M}^{\times}$ whose residue lies outside $k_{K}$. Since $k_{K}$ is algebraically closed, $\res(u)$ is transcendental over $k_{K}$, and hence $u$ is transcendental over $K$. Therefore the $K$-algebra morphisms
\[
\alpha\colon A\to M,\ x\mapsto u,\qquad\beta\colon A\to M,\ x\mapsto tu,
\]
are injective and $\Ker\alpha=\Ker\beta=(0)$. However $[\alpha]\neq[\beta]$: since the residue field of $M$ is algebraically closed, every unit is a square (\cref{lem:powers}), so $\mathbf{M}\models P_{2}(u)$, while $v(tu)=v(t)\notin2\Gamma_{M}$ and $\mathbf{M}\models\neg P_{2}(tu)$; an isomorphism $\gam_{0}$ as in \cref{def:T-spec} would have to send $u\mapsto tu$ and preserve $P_{2}$. A point of $\tSpec(A)$ thus carries strictly more information than its kernel, namely the $\cL$-structure induced on the residue field --- just as a prime of a real ring may support several orderings of its residue field.
\end{example}


\begin{remark}\label{rem:topologies}
Under the bijection $\Phi$ of \cref{thm:types}(2), the basic open $U(f,n)$, for $f=(f_{0},\dots,f_{m})$ a tuple of polynomials, corresponds to $\big[f_{0}(x)\neq0\wedge\bigwedge_{i}P_{n_{i}}(f_{i}(x))\big]$, a clopen set of $S_{n}(K)$. The spectral topology is thus the coarsening of the logic topology of $S_{n}(K)$ generated by these \emph{positive} conditions only, their negations being discarded --- exactly as the spectral topology of the real spectrum coarsens the constructible topology, whose clopen sets are given by arbitrary sign conditions. In particular, $\Phi$ is not a homeomorphism: the spectral topology is only $\mathrm{T}_{0}$ and has nontrivial specializations.
\end{remark}

\section{Open questions}\label{sec:questions}

We close with some directions for further work.

\subsection*{Real Laurent series}
The most immediate question is the analogue of our results for the field $\rlt$ of real Laurent series, a henselian valued field whose residue field $\R$ is real closed rather than algebraically closed. Here two structures interact: the valuation, controlled as above by the $t$-Kochen operator, and the order. Exploratory work carried out during the preparation of this article suggests the following statements, in the language $\cL_{t}$ of pointed rings (with the natural adaptations of our definitions: $T_{\R}=\Th_{\cL_{t}}(\rlt)$ and $\gam$, $\Lam_{L}$ unchanged).

\begin{conjecture}\label{conj:rt-1}
Let $\mathbf{K}\models T_{\R}$ and let $L$ be an ordered field extension of $K$. Then $L$ admits a structure making it a model of $(T_{\R})_{\forall}$ compatible with the order and with the structure of $\mathbf{K}$ if and only if the following two conditions hold:
\begin{enumerate}[label=\textnormal{(\arabic*)}]
\item $1/t\notin\cO_{K}[\gam(L)]$;
\item $\cO_{K}[\gam(L)]$ is a convex subring of $L$.
\end{enumerate}
\end{conjecture}

\begin{conjecture}\label{conj:rt-2}
Let $\mathbf{K}\models T_{\R}$ and let $I$ be an ideal of $\KX$. Then $J(Z(I))=I$ if and only if $I$ is real and $I=\rt[R]{I}$, where $\rt[R]{I}$ is defined as in \cref{def:T-radical}.
\end{conjecture}

The expected Nullstellensatz for $\rlt$ would then combine the real radical and the $t$-Kochen radical, in analogy with the way the $p$-adic Nullstellensatz of \cite{jardenroquette} refines to Kochen's theory in the presence of orderings.

\subsection*{Topology of the $T$-spectrum}
\Cref{thm:types} identifies $\tSpec(\KX)$ with a space of types endowed with a positive topology. It is natural to ask whether $\tSpec(A)$ is a spectral space in the sense of Hochster for every $K$-algebra $A$ (compare \cite{DST19}, and \cite{Belair90,Srhir08} for the $p$-adic case), whether the quasi-compact opens are exactly the finite unions of basic opens, and what the specialization order is. A related question, following Robinson's topos-theoretic construction of the $p$-adic spectrum \cite{Robinson86}, is whether $\tSpec$ underlies a Cole spectrum for a coherent theory of ``$T$-closed local rings''.


\bibliographystyle{amsplain}
\bibliography{biblio}

\end{document}